\documentclass[12pt]{amsart}
\usepackage{amsthm}
\usepackage{amsmath}
\usepackage[margin=1in]{geometry}
\usepackage{hyperref}
\begin{document}
\newcommand{\Q}{{\mathbb Q}}
\newcommand{\C}{{\mathbb C}}
\newcommand{\R}{{\mathbb R}}
\newcommand{\Z}{{\mathbb Z}}
\newcommand{\F}{{\mathbb F}}
\renewcommand{\wp}{{\mathfrak p}}
\renewcommand{\P}{{\mathbb P}} 
\renewcommand{\O}{{\mathcal O}}
\newcommand{\Pic}{{\rm Pic\,}}
\newcommand{\Ext}{{\rm Ext}\,}
\newcommand{\rank}{{\rm rk}\,}
\newcommand{\sbull}{{\scriptstyle{\bullet}}}
\newcommand{\bX}{X_{\overline{k}}}
\newcommand{\ch}{\operatorname{CH}}
\newcommand{\tors}{\text{tors}}
\newcommand{\cris}{\text{cris}}
\newcommand{\alg}{\text{alg}}
\let\isom=\simeq
\let\rk=\rank
\let\tensor=\otimes

\newtheorem{theorem}[equation]{Theorem}      
\newtheorem{lemma}[equation]{Lemma}          %
\newtheorem{corollary}[equation]{Corollary}  
\newtheorem{proposition}[equation]{Proposition}
\newtheorem{scholium}[equation]{Scholium}

\theoremstyle{definition}
\newtheorem{conj}[equation]{Conjecture}
\newtheorem*{example}{Example}
\newtheorem{question}[equation]{Question}

\theoremstyle{definition}
\newtheorem{remark}[equation]{Remark}

\numberwithin{equation}{section}

\newcommand{\be}{\begin{equation}}
\newcommand{\ee}{\end{equation}}
\newcommand{\bes}{\begin{equation*}}
\newcommand{\ees}{\end{equation*}}
\newcommand{\bea}{\begin{eqnarray}}
\newcommand{\eea}{\end{eqnarray}}
\newcommand{\beas}{\begin{eqnarray}}
\newcommand{\eeas}{\end{eqnarray}}
\newcommand{\bth}{\begin{theorem}}
\newcommand{\eth}{\end{theorem}}
\newcommand{\bpro}{\begin{prop}}
\newcommand{\epro}{\end{prop}}
\newcommand{\bp}{\begin{proof}}
\newcommand{\ep}{\end{proof}}
\newcommand{\blem}{\begin{lemma}}
\newcommand{\elem}{\end{lemma}}
\newcommand{\br}{\begin{remark}}
\newcommand{\er}{\end{remark}}

\title{Some remarks on Hodge symmetry}
\author{Kirti Joshi}
\address{Math. department, University of Arizona, 617 N Santa Rita, Tucson
85721-0089, USA.}
\email{kirti@math.arizona.edu}

\begin{abstract}
I make some remarks on Hodge symmetry, and prove for instance that if $X/k$ smooth, proper and Hodge-Witt, Hodge de Rham sequence of $X$ degenerates at $E_1$ and $X$ has torsion-free crystalline cohomology, then Hodge symmetry holds for $X$.
\end{abstract}

\maketitle
\begin{quote}
\hfill 
\end{quote}


\section{Introduction}
Let $X/k$ be a smooth proper variety. Suppose that the crystalline cohomology of $X$ is torsion-free and the Hodge-de Rham spectral sequence of $X$ degenerates at $E_1$.
For $p,q\geq 0$, let
$h^{p,q}=\dim H^q(X,\Omega^{p}_{X/k})$. I say that $X$ satisfies Hodge symmetry if 
\be\label{Hodge-symmetry}
\dim_kH^q(X,\Omega^{p}_{X/k})=
\dim_kH^p(X,\Omega^{q}_{X/k}).
\ee 
When $k=\C$, by the Hodge
decomposition theorem (see \cite[Chap 0, Section
7]{griffiths-harris}) 
these spaces  are complex conjugates and hence Hodge symmetry holds over complex numbers. 

        In this brief note I give an elementary, inasmuch as a proof using the principal results of \cite{deligne80, illusie79b, illusie83b, ekedahl-diagonal} can be construed to be elementary, of Hodge symmetry, when $X/k$ is a smooth, proper and Hodge-Witt with torsion-free crystalline cohomology over a perfect field $k$ of characteristic $p>0$.

		One can summarize my approach as follows. In \cite{ekedahl-diagonal} Ekedahl characterized the highest Hodge polygon (i.e., a convex polygon with integer slopes and break points) which lies below the Newton polygon (of $H^*_{cris}(X/W)$). This polygon is given explicitly as the slope-number polygon, so the Newton polygon lies on or above the slope-number polygon which lies on or above the Hodge polygon. My remark is that under suitable hypothesis on $X$, the Hodge and slope-number polygon coincide and Hodge symmetry reduces to slope-number symmetry. In \cite{ekedahl-diagonal} it was shown that slope number symmetry holds for any variety dominated by a smooth, projective variety (by \cite{deligne74,katz74}), and hence via Chow's Lemma and Abhyankar's resolution, for smooth, proper threefolds (with $p\geq 5$) using \cite{deligne80}. Let me remind the reader that Hodge symmetry implies that Betti numbers (i.e. de Rham numbers) are even in odd cohomological dimension. This property was proved  in \cite[Remark, page 112]{ekedahl-diagonal} for smooth, projective varieties and also for smooth, proper threefolds. Ekedahl used Chow's Lemma and resolution of singularities, to go from projective to proper, and hence he assumed dimension three (and $p\geq 5$). Recently this property of Betti numbers has been extended to smooth, proper case in all dimensions, independently in \cite{suh12}.  The proof of \cite[VI, Lemma 3.1(ii)]{ekedahl-diagonal} and \cite[Theorem 2.2.2]{suh12} are identical except for the replacement of resolution by de Jong's Theorem and so at any rate slope-number symmetry also extends to the smooth, proper case.
		
		 The technical tool which allows one to go between Hodge numbers and slope-numbers are the Hodge-Witt numbers (which can be negative, in general, unlike the Hodge numbers). One may think in terms of the following schematic (modulo Hodge de Rham degeneration and torsion-freeness of crystalline cohomology of $X$):
\bes\let\trm=\textrm\let\imp=\Longrightarrow
	\trm{Hodge-Witt+slope-number Sym}\imp\trm{Hodge-Witt Sym}\imp\trm{Hodge Sym}.
\ees	
		
		In \cite{joshi04} I proved Hodge symmetry holds for any smooth, projective threefold (whose Hodge de Rham spectral sequence degenerates and which have torsion-free crystalline cohomology), and also conjectured that Hodge-Witt symmetry holds under a suitable hypothesis. These results did not require that $X$ be Hodge-Witt. (It is clear upon examining that proof that one may replace projective by proper without affecting its conclusion.)

A smooth, proper variety over a perfect field $k$ is Hodge-Witt if $H^i(W\Omega_X^j)$ are finite type $W=W(k)$-modules for all $i,j\geq 0$ (here $W(k)$ is the ring of Witt-vectors of $k$). By \cite[Th\'eor\`eme 3.7]{illusie79b} this is equivalent to the degeneration of the slope spectral sequence at $E_1$. Ordinary varieties are Hodge-Witt. Examples of Hodge-Witt varieties are easy to construct.  Let $Y,Z/k$ be smooth, proper varieties and assume that $Y$ is ordinary, then $X=Y\times Z$ is Hodge-Witt. Any curve, any K3 surface of finite height is Hodge-Witt (on the other hand a K3 surface of infinite height is not Hodge-Witt); and so is any abelian variety of dimension $g$ and of $p$-rank at least $g-1$. 
Blowup of a smooth, proper, Hodge-Witt variety along a closed, Hodge-Witt subvariety is also Hodge-Witt. Finally note that the class of ordinary varieties is a proper subclass of Hodge-Witt varieties. 

In general, even if Hodge de Rham degenerates at $E_1$ and crystalline cohomology is torsion-free, the numbers $h^{i,j}$ do not agree with the slope-numbers $m^{i,j}$, so one cannot use symmetry of slope-numbers to deduce Hodge symmetry. 

To deduce Hodge symmetry for any smooth, proper $X/\C$ from the above characteristic $p>0$ argument, one needs the existence of one prime $p>\dim(X)$ at which $X$ has  good, Hodge-Witt reduction. Unfortunately this is not known, but I certainly expect it to be true. In fact in \cite{joshi00} C.~S.~Rajan and I have conjectured that there are infinitely many such primes. Jean-Pierre Serre has conjectured that there exist infinitely many primes of ordinary reduction (Serre's conjecture predates the conjecture Rajan and I make); but one does not know if there is even one prime $p>\dim(X)$ of ordinary or Hodge-Witt reduction.

\section{Recollections}
        To keep this note brief, I will refer to \cite{illusie79b},
\cite{illusie83b} and \cite{ekedahl84}, \cite{ekedahl85},
\cite{ekedahl-diagonal} for notations and basic results. In particular
I do not recall the notion of dominos here but make use of it.  
 In this section $X/k$ is a smooth
proper variety over a perfect field of characteristic $p>0$. Let
$H^j(X,W\Omega^i_X)$ be the Hodge-Witt cohomology groups of $X$. Let
$T^{i,j}$ be the dimension of the domino (see \cite[page
42]{illusie83a}) associated to the differential
$$H^j(X,W\Omega^i_X)\to H^j(X,W\Omega^{i+1}_X).$$ Let $m^{i,j}$ be the
slope numbers (see \cite[0, 6.1 and 6.2 (ii)]{ekedahl-diagonal}) associated to the slopes of Frobenius on
the crystalline cohomology of $X$. I recall the definition here for
the reader's convenience.  Let for any rational number $\lambda$, let
$h^n_{\cris,\lambda}$ be the dimension (=multiplicity) of the slope
$\lambda$ in $H^n_{\cris}(X/W)$. Then by definition of $m^{i,j}$ one has
\begin{equation}
m^{i,j}=\sum_{\lambda\in [i,i+1)}(i+1-\lambda)h^{i+j}_{\cris,\lambda}
        +\sum_{\lambda\in [i-1,i)}(\lambda-i+1)h^{i+j}_{\cris,\lambda}
\end{equation}

Then the
Hodge-Witt numbers of $X$ (see \cite[page 64]{illusie83a}), denoted
$h^{i,j}_W$, are defined to be
\begin{equation}\label{Hodge-witt-numbers}
h^{i,j}_W=m^{i,j}+T^{i,j}-2T^{i-1,j+1}+T^{i-2,j+2}.
\end{equation}

    Note that by \cite[I, 2.18.1]{illusie83b} $T^{i,j}$ is zero if the
corresponding differential of the slope spectral sequence is zero.

    The following symmetry of slope numbers is a consequence of
\cite{deligne80} and \cite[Theorem 1]{katz74} and is due to Ekedahl (see
\cite[VI, 3.1 (ii)]{ekedahl-diagonal}) for $X$ projective.
\begin{lemma}\label{slope-number-symmetry}
For any smooth, proper variety $X/k$ over a perfect field $k$ of
characteristic $p$ and for all $i,j$ one has
\begin{equation}
m^{i,j}=m^{j,i}.
\end{equation}
\end{lemma}

\bp
As was pointed out in the Introduction, in \cite{ekedahl-diagonal} it was shown that slope number symmetry holds for any variety dominated by a smooth, projective variety (by \cite{deligne74,katz74}). In particular it holds for smooth, projective varieties. Further it was shown using Chow's Lemma and Abhyankar's resolution, for smooth, proper threefolds (with $p\geq 5$) and \cite{deligne80} that the assertion also holds for smooth, proper threefolds. It is clear that one may replace resolution by de Jong's Theorem \cite{dejong96} (as was pointed out in \cite{suh12}). To prove the proper case, one needs to prove that various properties of crystalline Frobenius (in the projective case) extend to the smooth, proper case. Specifically one needs to verify that if $\lambda$ is a slope of Frobenius on $H^i_{cris}(X/W)$ then $i-\lambda$ is also a slope and of the same multiplicity as $\lambda$. This is proved in \cite[VI, Lemma 3.1(ii)]{ekedahl-diagonal} for projective $X$ and for $X$ proper with $\dim(X)\leq 3$ by Chow's lemma and resolution; and for $X$ proper  in \cite[Theorem~2.2.2]{suh12}. It should be noted that the proof of \cite[Theorem~2.2.2]{suh12} and that of slope-number symmetry \cite[VI, Lemma 3.1(ii)]{ekedahl-diagonal} are identical except for the use of de Jong's Theorem on alterations in the former and Abhyankar's resolution theorem in the latter.
\ep

\section{Hodge Symmetry}
\begin{theorem}\label{main2}
    Let $X/k$ be a smooth proper variety  over a perfect field $k$ of characteristic $p>0$.
Assume that 
\begin{enumerate}
\item $X$ is Hodge-Witt,
\item the crystalline cohomology of $X$ is torsion free,
\item Hodge de Rham spectral sequence of $X$ degenerates at $E_1$,
\end{enumerate} 
Then Hodge
symmetry \eqref{Hodge-symmetry} holds for $X/k$.
\end{theorem}

\begin{proof}


    One first notes that $H^*(X,W\Omega^\sbull_{X})$ is a
Mazur-Ogus object in the derived category of bounded complexes of
modules over the Cartier-Dieudonne-Raynaud algebra. To see this
one observes that the slope spectral sequence computes the
crystalline cohomology of $X$ (see \cite[3.1.1, page
614]{illusie79b}) and that $H^*(X,W\Omega^\sbull_{X})$ is a
coherent module over the Cartier-Dieudonne-Raynaud algebra (see
\cite[II, Theorem 2.2]{illusie83b}). Since the
crystalline cohomology of $X$ is torsion free, so  by the
universal coefficient theorem one sees that:
$$\text{rank}_W H^n_{cris}(X/W)=\dim_k H^n_{dR}(X/k).$$ Finally as the
Hodge-de Rham spectral sequence of $X$ degenerates at $E_1$, one sees
that the number on the right is equal to $\sum_{i+j=n}\dim
H^i(X,\Omega^j_X)$.
Hence by the definition of Mazur-Ogus objects (see
\cite[IV, 1.1]{ekedahl-diagonal}) one sees that
$H^*(X,W\Omega^{\sbull}_X)$ is a Mazur-Ogus object. Hence one can
apply \cite[Corollary~3.3.1, page 86]{ekedahl-diagonal} to see that
$h^{i,j}_W=h^{i,j}$. 

Now the hypothesis that $X$ is Hodge-Witt means that the slope spectral sequence of
$X$ degenerates at $E_1$. In this event all the domino numbers $T^{i,j}=0$ and 
hence the Hodge-Witt numbers are given simply in terms of the slope numbers:
\be 
h_W^{i,j}=m^{i,j}.
\ee
Next I observe that Hodge-Witt symmetry holds for $X$:
\be 
h^{i,j}_W=h^{j,i}_W.
\ee
This is a consequnce of the symmetry of slope numbers (see Lemma~\ref{slope-number-symmetry}).
Hence one has 
\bes 
	h^{i,j}=h_W^{i,j}=h_W^{j,i}=h^{j,i}.
\ees This proves my assertion.
\end{proof}

\section{Hodge symmetry over $\C$}
Now assume $k=\C$ and $X$ is smooth, proper over $\C$.  
Unfortunately one cannot use Theorem~\ref{main2} to deduce the Hodge symmetry holds over complex numbers, as I do not know if there always exists a prime $p$, of good reduction for $X$, such that the reduction of $X$ is Hodge-Witt with torsion-free crystalline cohomology. In \cite{joshi00} C.S.~Rajan and I have conjectured that in fact there are infintely such primes. Jean-Pierre Serre has conjectured that there are infintely many primes of ordinary (and therefore Hodge-Witt) reduction. Let me record the  what can be proved under weakest possible hypothesis on $X$:

\begin{theorem}\label{main}
Let $X/\C$ be a smooth proper variety over complex numbers.  Suppose that there exists one prime $p>\dim(X)$ such that 
\begin{enumerate}
\item $X$ has good, Hodge-Witt reduction at $p$,
\item the de Rham cohomology of $X$ is torsion-free at $p$.
\end{enumerate}
Then Hodge symmetry
\eqref{Hodge-symmetry} holds for $X$.
\end{theorem}

\begin{proof} (of Theorem~\ref{main})
This is done by a standard argument which allows one to reduce to Theorem~\ref{main2} and will be omitted.    
\end{proof}

\br 
Suppose $\dim(X)=n$ and $X/\C$ is projective. Serre duality gives $h^{i,n-i}=h^{n-i,i}$. So Hodge symmetry for $X$ is a non-trivial assertion for cohomology in dimension $\leq n-1$; and moreover Hodge symmetry is always true for $H^1(X)$ (Weil for curves or by reduction to abelian varieties) at any rate one can prove Hodge symmetry for $H^1(X)$ purely algebraically. Hence one deduces (purely algebraically) that Hodge symmetry holds for a surface. If Hodge symmetry is established for any smooth, projective variety of dimension $\leq n-1$, then by Lefschetz hyperplane section theorem, one deduces Hodge symmetry holds for a variety of dimension $n$ but in cohomology dimension $< n-1$ (one notes that by Deligne-Illusie,  Kodaira vanishing over $\C$  can be proved algebraically, and hence (weak) Lefshetz). So any inductive argument will eventually have to deal with $H^{n-1}(X)$ algebraically. This I do not know how to do at the moment.
\er


\end{document}